\documentclass[paper,12pt]{amsart}
\usepackage{subfigure}
\usepackage{amsfonts}
\usepackage{newlfont}

\usepackage[centertags]{amsmath}
\usepackage{amsfonts}
\usepackage{amssymb}
\usepackage{graphicx}
\usepackage{footnote}
\usepackage{mathrsfs}
\usepackage{makecell}
\usepackage{enumerate}
\usepackage{tikz}
\usepackage{tikz-3dplot}
\usepackage{float}
\usetikzlibrary{calc}

\numberwithin{equation}{section}
\newtheorem{theorem}{Theorem}[section]
\newtheorem{definition}[theorem]{Definition}
\newtheorem{proposition}[theorem]{Proposition}

\newtheorem{corollary}[theorem]{Corollary}
\newtheorem{rmk}[theorem]{Remark}

\newtheorem{example}[theorem]{Example}

\newcommand{\R}{\mathbb R}

\newcommand{\C}{\mathbb C}

\usepackage{hyperref}
\hypersetup{
    colorlinks=true,
    linkcolor=blue,
    filecolor=magenta,      
    urlcolor=cyan,
    pdfpagemode=FullScreen,
    }

\title[Existence and uniqueness of radial solutions]{Existence and uniqueness of radial solutions of semi-linear equations on manifolds}

\author{Nicolas Martinez-Alba}
\address{Departamento de Matem\'aticas, Universidad Nacional de Colombia, Bogot\'a, Colombia. Carrera 30 Calle 45, Ciudad Universitaria, Bogot\'a, Colombia}
\email{nmartineza@unal.edu.co}

\author{Oscar Ria\~no}
\address{Departamento de Matem\'aticas, Universidad Nacional de Colombia, Bogot\'a, Colombia. Carrera 30 Calle 45, Ciudad Universitaria, Bogot\'a, Colombia}
\email{ogrianoc@unal.edu.co}

\begin{document}


\begin{abstract} 
We investigate the existence and uniqueness of solutions for second-order semi-linear partial differential equations defined on a Riemannian manifold $M$. By combining differential geometry and analysis techniques, we establish the existence and uniqueness of constant solutions through the orbits of a group action. Our approach transforms such problems into equivalent ones over a submanifold $\Sigma$ of dimension one, which is transversal to the group action. This reduction leads us to a one-dimensional setting, where we can apply different results from the theory of ordinary differential equations. Our framework is versatile and includes the setups of polar actions or exponential coordinates, with particular examples such as the sphere, surfaces of revolution, and others.

\end{abstract}

\maketitle

\tableofcontents


\section{Introduction}

We consider the semi-linear elliptic problem 
\begin{equation}\label{EP1}
    -\Delta_M u +f(x,u(x))=0 \quad \text{ in } \, M
\end{equation}
posed on a $n$-dimensional manifold $M$, where $\Delta_M$ stands for the Laplace-Beltrami operator defined on a Riemannian manifold $(M,g)$, and the function $f$ represents the nonlinearity. Semi-linear equations as \eqref{EP1} arise in diverse applications including fluid dynamics, electromagnetism, quantum mechanics, optics, image processing, and geometry, among many others. They also emerge as stationary states of nonlinear evolution equations, such as the heat and wave equations. In optics and the study of water waves, they appear when examining ground-state solutions of the nonlinear Schr\"odinger equation. Additionally, in the context of Bose-Einstein condensates, the time-independent Gross-Pitaevskii equation also provides an example of \eqref{EP1}.

In this manuscript, we seek to establish different results on the existence and uniqueness of radial solutions of the partial differential equation (PDE) in \eqref{EP1}. The literature on such questions for elliptic models on manifolds is quite extensive, we will only refer to the works in \cite{BahriBrazis1996,BecerraGalvisMartinez2019,CastroJacobsen2023,Cherrier1984,ManciniSandeep2008,PigolaRigoliSett2010}. However, even when solutions are shown to exist, uniqueness is not always guaranteed in general, e.g., see \cite{AmbrosettiRabinowitz1973,BartschWillem1993,CastroJacobsen2023,CastroKurepa1987} and references therein.

Another important topic in the theory of elliptic equations is the study of Unique Continuation Principles (UCP), which appear in various contexts and applications. Common techniques in this regard include Carleman-type estimates and doubling inequalities. For a non-exhaustive list of references, we refer to \cite{Aronszajn1957,BarceloKenigRuizSogge1988,GarafaloLin1987,Hormander1985,RoudJeroLebeauRobb2022,Lerner2019,Tataru1995,Taylor1981}. 

We investigate the existence and uniqueness of solutions using geometrical and analytical techniques based on assumptions on the space where the problem is defined and the nonlinear term. More precisely, we first examine conditions on $M$ and $f$ to assure the existence of radial solutions for problem \eqref{EP1}. In this context, a radial function  $u$ is a function that only depends on a suitable codimension 1 submanifold  $\Sigma\hookrightarrow M $. We refer to Definitions~\ref{def:radialfunction} below for more details on this concept. Next, we focus on establishing the uniqueness of radial solutions in the following sense:
\\ \\ \emph{Suppose that $u_1$, $u_2$ are sufficiently regular radial solutions of \eqref{EP1} such that $u_1(p)=u_2(p)$, and $\nabla u_1(p)=\nabla u_1(p)$ for some $p$ in the codimension 1 submanifold  $\Sigma\hookrightarrow M $. Then $u_1\equiv u_2$ in all $M$.}
\\ \\
Note that if the above uniqueness result holds, we have that equality of two solutions of \eqref{EP1} in an open $\Omega \subset M$ that intercepts $\Sigma$, forces both solutions to be equal in $M$, i.e., we obtain a version of the following UCP:
 \\ \\
\emph{Suppose that $u_1$, $u_2$ are sufficiently regular radial solutions of \eqref{EP1} such that $u_1=u_2$ in some open set $\Omega \subset M$, then $u_1\equiv u_2$ in all $M$.} 
\\ \\
As previously emphasized, our studies of existence and uniqueness will focus on manifolds $M$ equipped with suitable geometry that enable us to simplify the Laplace-Beltrami operator, along with certain regularity conditions on the radial nonlinear term $f (x, y)$.  This approach allows us to transform existence, uniqueness, and UCP problems for \eqref{EP1} into equivalent statements for an initial value problem of an ordinary differential equation (ODE). It is worth mentioning that our class of geometries is big enough to consider special cases, such as those of constant sectional curvature, e.g., positive (sphere), flat (Euclidean space), or negative curvature (hyperbolic spaces). For the proof of our main results, we adopt the same strategy introduced in \cite{BecerraGalvisMartinez2019}, where the existence of nonnegative radial solutions of the problem \eqref{EP1} with nonlinearity $f(x,y)=b(x)g(y)$ was studied. The novelty of this paper lies in the introduction of different classes of analytical conditions on the geometry (associated with the volume of the orbits) and on the nonlinearity term $f(x,y)$ (e.g., being locally Lipschitz, differentiable, and with certain integrability conditions). These new analytical and geometrical considerations, together with the geometric manipulation of the Laplace-Beltrami operator (the radial part as in \cite{BecerraGalvisMartinez2019}) lead us to establish various existence and uniqueness results, which have not been addressed in the literature of semi-linear problems.


\subsection*{Organization of the document}

The paper is organized as follows: Section \ref{mainReslSect} contains the main results presented in this paper. More precisely, in Subsection \ref{AssumSection}, we establish the geometric and analytic conditions under which our results hold. In Subsection \ref{mainresulSubs},  we state our main results. We conclude Section \ref{mainReslSect} by providing explicit examples of manifolds $M$ and nonlinearities $f$ where our results are applicable. The main goal of Section~\ref{sec:geometry} is to further analyze examples relevant to this paper, and to discuss the geometric assumptions, as well as the radial part of the Laplace–Beltrami operator along submanifolds. In Section \ref{sec:existe}, we use our geometric reductions and results from the theory of ODEs to show the existence of solutions of problem \eqref{EP1}, i.e., we establish Theorem \ref{TheoMainExis} as well as the nonexistence result in Proposition \ref{propnonexis}. In Subsection \ref{AddGeconstr}, we include geometrical techniques for vector fields to establish further existence results. Additionally, we use the particular geometry of two-point homogeneous spaces to extend solutions to singular points of the Laplace-Beltrami operator. Finally, in Section \ref{sec:uniq}, we show our uniqueness results in Theorem \ref{TheoMainUniq}.


\section{Statement of the main results}\label{mainReslSect}

\subsection{Assumptions}\label{AssumSection}

In this section, we present some geometrical and analytical assumptions needed to assure the existence and uniqueness of solutions for problem \eqref{EP1}. 


\subsubsection{Geometrical assumptions}\label{sec:geom assumption}

Along the manuscript, we will assume  $M$ to be a connected, paracompact, Hausdorff $n$-dimensional manifold (possible with boundary). In particular, $M$ can be endowed with a Riemannian metric $g$. From now on, we will consider a Riemannian manifold $(M,g)$ with the following additional assumptions:

\begin{enumerate}

\item[{[G1]}]
Consider a {\bf polar action} on $M$, that is, a Lie group $G$ acting on $M$ by isometries and an immersed connected submanifold $\tilde{\Sigma}$ in $M$ that  
 is transversal to the non-trivial orbits along interior points, i.e., for each $x\in \tilde{\Sigma}\backslash \partial M$  with $G\cdot x\neq \{x\}$, we have 
$$T_xM=T_x\tilde{\Sigma}\oplus T_x(G \cdot x).$$

\end{enumerate}
The key idea of this paper is to reduce the problem \eqref{EP1} to $\tilde{\Sigma}$, using $G$-invariant functions. But to work with a well-posed problem, in addition to [G1], we need some extra conditions on $\tilde{\Sigma}$.

\begin{enumerate}
\item[{[G1(a)]}] 
For this assumption, $M$ has an empty boundary and $\Sigma$  will denote the set of non-fixed points of $\tilde{\Sigma}$ by the action. We assume that $\Sigma$ is an open, embedded, connected submanifold of $\tilde{\Sigma}$. In addition,  the orbits  $G\cdot x$ for $x\in \Sigma$ have finite Riemann measure $A(x)$ in $M$, for which the Riemann volume function $A:\Sigma\to \R^+$ is smooth. Furthermore, we will assume that the set of not fixed points
\begin{equation}\label{Mfixdef}
    M_0=\{x\in M: G\cdot x\neq\{x\}\}
\end{equation}
is open in $M$.

\item[{[G1(b)]}]  For this assumption, $M$ has boundary and $\Sigma$ denotes the set of non-border and non-fixed points on $\tilde{\Sigma}$. We assume that $\Sigma$ is an open, embedded, connected submanifold of $\tilde{\Sigma}$. In addition,  the orbits  $G\cdot x$ for $x\in \Sigma$ have finite Riemann measure $A(x)$ in $M$, for which the Riemann volume function $A:\Sigma\to \R^+$ is smooth. Furthermore, we will assume that the set
\begin{equation}\label{Mfixdef2}
    M_0=\{x\in M: G\cdot x\neq\{x\}\}\backslash \partial M
\end{equation}
is open in $M$.
\end{enumerate}

Note that, in both cases, the $G$ action on $\Sigma$ defines the whole $M_0$.

The main results in this manuscript apply to {\bf unimodular $G$-actions}, that is the Lie group in condition [G1] must be unimodular. We recall that any \emph{compact}, \emph{abelian, semi-simple}, or \emph{nilpotent Lie group} is unimodular\footnote{A Lie group is called {\bf unimodular} if every left Haar measure is a right Haar measure, and vice versa. For more details, see \cite[Section~VIII, Chapter~3]{Knapp2002}.}. For simplicity, all the groups $G$ considered in this manuscript are either compact or abelian.

It is worth mentioning that we are considering $A$ to be the Riemannian measure of the orbits, thus $A$ vanishes at fixed points of the $G$-action.  The finiteness condition on $A$, together with the special case of unimodular $G$-actions, naturally leads us to work with the {\it radial part} of the Laplace operator 
(see \cite[Theorem 2.11]{HelgasonIII} or Theorem \ref{t8} below).  

\begin{rmk}
Any compact group $G$ provides an example of an unimodular group. Furthermore, it is known that the compactness of $G$ implies that the orbits are compact (and thus, $A$ is finite), and that $M_0$ is open and dense. This follows from the \emph{Principal Orbit Theorem} \cite{Tammo1987}, which concerns group actions and symmetric spaces. Accordingly, for compact groups, the conditions [G1(a)] and [G1(b)] are always satisfied.
\end{rmk}

Corollary \ref{cor:reduction to ODE} below is our main tool to study problem \eqref{EP1}, reducing its dimension to that of the transversal submanifold. Specifically, we focus on the instance where $\tilde{\Sigma}$ is a one-dimensional submanifold, with polar actions that generalize polar coordinates. Furthermore, in certain cases, using a global chart for $\Sigma$ will be useful, which can be achieved by considering the {\it arclength parametrization}. This allows the introduction of an appropriate change of coordinates:

\begin{enumerate}
      
    \item[{[G2]}] Let us choose $r_0$ a  point in  $(0,\mbox{lenght}(\Sigma))$, and assume that there exist $c_1,c_2$ with $-\infty\leq c_1<c_2\leq \infty$ such that the change of variables $s=J(r):= \int_{r_0}^r \frac{1}{A(t)}\, dt$  maps some interval $(a,b)\subset [0,\mbox{lenght}(\Sigma))$ into $  (c_1,c_2)$ as a diffeomorphism of class $C^2$.

\end{enumerate}

Our main goal is to construct a specific solution to the semi-linear problem \eqref{EP1} after fixing the manifold $M$ and the nonlinear terms $f$. We will focus on a particular family of solutions that is invariant under the polar action. To clarify this concept, we provide the definition of a (local) radial solution.

\begin{definition}\label{def:solution}\label{def:radialfunction}

For a $C^k$ {\bf radial function}, we mean that $u\in C^k(M)$ with  $u(g\cdot p)=u(p)$ for each pair $(g,p)\in G\times M$. For a $C^k$  {\bf local radial solution} (at $p\in M$), we mean that $u\in C^k(U_e\cdot  U_p)$ with $U_p$ be an open set of $p$ in $\Sigma$, and $U_e$ be an open of $e\in G$ such that
\begin{itemize}
    \item $u(g\cdot p)=u(p)$ for each pair $p,g\cdot p\in U$,
    \item and $u$ solves \eqref{EP1} in $U$.
\end{itemize}
 We will say that $u$ is a $C^k$  {\bf maximal radial solution} if $u\in C^k(M_0)$ solves \eqref{EP1}.
\end{definition}

The above definition is motivated by the fact that we cannot guarantee a solution on fixed points of the action, because the Laplace--Beltrami operator may have singularities at those points. For example, in the Euclidean case, the usual change to polar coordinates leads to a singularity at the origin. 


\subsubsection{Analytical assumptions}

Under condition [G1], we can reduce the dimension of the PDE \eqref{EP1}, leading to an ODE. Thus, we establish conditions on the nonlinear term $f$ and the Riemann measure function $A$ that assure existence and uniqueness results for the original problem \eqref{EP1}. 
\\ \\
We will work with the following conditions.
\begin{itemize}
    \item[{[F0]}] $f\in C(M\times \mathbb{R})$ is a radial function in the  first variable.
\end{itemize}
As frequently used in classical ODE theory, we will assume that the function $f$ above is \emph{locally Lipschitz} in the following sense:
\begin{itemize}
\item[{[F1]}] The function $f(x,y)$ is locally Lipschitz continuous in $y$ (or in the second variable), if given $K\subset M \times \mathbb{R}$ compact, there exists $C_K>0$ such that
\begin{equation*}
    |f(x,y_1)-f(x,y_2)|\leq C_K|y_1-y_2|,
\end{equation*}
for every $(x,y_1),(x,y_2)\in K$.
\end{itemize}

Once we are in a one-dimensional context, it is natural to consider standard conditions for the global existence of solutions for ODEs (for more details, we refer to \cite{CoddingtonLevinson1955,Constantin1995,Sideris2013,Gerald2012}). 

\begin{itemize}
\item[{[F2]}] Assume the change of variables in condition [G2]. For all $(s,y)\in (c_1,c_2)\times \mathbb{R}$, suppose
\begin{equation*}
    |f(r(s),y)|\leq L_1(s)+L_2(s)|y|,
\end{equation*}
where $L_1$, $L_2$  are locally integrable functions on $(c_1,c_2)$, which are nonnegative almost everywhere.
\end{itemize}

Assuming [G2] with $c_1,c_2$ being finite numbers, we can obtain solutions of \eqref{EP1} for possible singular nonlinearities $f$. Inspired by the results in \cite{GaticaOlikerWaltman1989} (see also \cite{AnuradhaHaiShivaji1996}), we work with  the following conditions:

\begin{itemize}
    \item[{[A1]}] The function $A(r(s))\in C([c_1,c_2])$ with $A>0$ in $(c_1,c_2)$. 
    \item[{[F3]}] We assume that in polar coordinates, $f(r(s),y)$, $f: (c_1,c_2)\times (0,\infty)\rightarrow (0,\infty)$ is continuous, $y \mapsto f(r(s),y)$ is decreasing for each $s\in (c_1,c_2)$, and $s \mapsto f(r(s),y)$ is integrable for each $y$.  
     \item[{[F4]}] Given $f$ as in [F3], we assume $$\lim_{y\to 0^{+}}f(r(s),y)=\infty \, \, \text{and } \, \, \lim_{y \to \infty}f(r(s),y)=0,$$ 
     uniformly on compact subsets of $(c_1,c_2)$.
     \item[{[F5]}] Consider $A$ as in [A1], and $f$ as in [F3]. We assume that 
     \begin{equation*}
         0<\int_{c_1}^{c_2} A(r(s))^2f\big(r(s),g_{\theta}(s)\big)\, ds<\infty, 
     \end{equation*}
     for all $\theta>0$, where $g_\theta (s):=\frac{\theta}{c_2-c_1}\min\{(s-c_1), (c_2-s)\}$.
\end{itemize}

In the case $\Sigma$ being parameterized by the interval $(a,\infty)$, for some $a\in \mathbb{R}$, i.e., the polar coordinates $r\in (a,\infty)$, following \cite[Theorem 1]{MaagliMasmoudi2001} (see also \cite{BacharMaagli2014,Zhao1994}), we can find positive solutions of \eqref{EP1} with some special asymptotic behavior at $r=a$, $r\to \infty$. We consider the conditions:
\begin{itemize}
    \item[{[A2]}] The function $A\in C([a,\infty))\cap C^1((a,\infty))$ with $A>0$ in $(a,\infty)$. $A^{-1}$ is integrable in a neighborhood of $a$.    
    \item[{[F6]}] We assume that in polar coordinates $f: (a,\infty)\times (0,\infty)\rightarrow \mathbb{R}$ is a measurable function, continuous at the second variable and such that
    \begin{equation*}
        |f(r,y)|\leq yh(r,y),
    \end{equation*}
    where $h$ is a nonnegative measurable function on $(a,\infty)\times (0,\infty)$ nondecreasing with respect to the second variable and such that
    \begin{equation*}
        \lim_{y\to 0^{+}} h(r,y)=0.
    \end{equation*}
Moreover,  suppose that
    \begin{equation*}
    \int_a^{\infty} A(r)\rho(r)h(r,\rho(r))\, dr<\infty,
\end{equation*}
where
\begin{equation}\label{rhodef}
    \rho(r)=\int_a^r A(t)^{-1}\, dt.
\end{equation}
\end{itemize}


\subsection{Main results}\label{mainresulSubs}

Let us now present the main results in this manuscript. 

\subsubsection*{Existence results}  Recall that our notion of maximal solution in Definition \ref{def:solution} corresponds to the fact that by passing to polar coordinates, we may have singularities at some fixed points of the action. Thus,  under the hypotheses in Subsection \ref{AssumSection} above, we can deduce different existence results for \eqref{EP1}, which we summarize in the following theorem. 

\begin{theorem}\label{TheoMainExis} 
Let $M$ be a $n$-dimensional Riemannian manifold satisfying condition [G1] for a 1-dimensional submanifold $\tilde{\Sigma}$ and an unimodular Lie group $G$, so that $\Sigma$ satisfies either conditions [G1(a)] or [G1(b)] 
\\ \\
i) Assume that $f$ satisfies [F0], [F1]. Then  for each $p\in \Sigma$ there exists a local radial solution to \eqref{EP1} of class $C^2$ in a neighborhood of $p$ in $M$.\\ \\
ii) Assume [G2], [F0], [F1], and [F2]. Then there exists a $C^2$ maximal radial solution of \eqref{EP1}. \\ \\
iii) Assume [G2] with $-\infty<c_1<c_2<\infty$. Suppose that the Riemann measure function $A$ satisfies [A1], and $f$ satisfies [F3], [F4] and [F5]. Then there exists a $C^2$ maximal radial solution of \eqref{EP1}, which is positive in $M_0$. Moreover, if $\Sigma$ is dense in $\tilde{\Sigma}$, then $u$ is in the class $C^1(M)$, that is, $u$ extends to the whole manifold $M$ but solves \eqref{EP1} over $M_0$. \\ \\

iv) Assume that in polar coordinates $\Sigma$ is parameterized by an interval $(a,\infty)$ for some $a\in \mathbb{R}$, i.e., the polar coordinates $r\in (a,\infty)$. Suppose that $A$ satisfies [A2], and that $f$ satisfies [F6]. Then there exists $u$ a $C^1$ maximal solution of \eqref{EP1}, which is positive in $M_0$, and in polar coordinates 
\begin{equation*}
    \lim_{r\to a^{+}} u|_\Sigma(r)=0, \quad \lim_{r\to \infty} \frac{u|_\Sigma(r)}{\rho(r)}=c 
\end{equation*}
for some $c>0$.\\

v) Assume that in polar coordinates $\Sigma$ is parameterized as  $(a,\infty)$ for some $a\in \mathbb{R}$. Suppose that $A$ satisfies [A2], and that $\lim_{r\to \infty}\rho (r)=\infty$, where $\rho$ is defined as in \eqref{rhodef}. We also assume:
\begin{itemize}
    \item[{[B1]}] Let $b\in C(M)$ be a radial function such that in polar coordinates $b\geq 0$, $b\neq 0$ and
\begin{equation*}
    \int_a^{\infty} A(t)\rho(\min \{t,1+a\})b(t)\, dt<\infty.
\end{equation*} 
\end{itemize}
Then, for any fixed $\sigma>0$ there exists a $C^2$ radial maximal solution of \begin{equation}\label{NegativeSigmaproblem}
   -\Delta_M u +b(x)u^{-\sigma}=0,
   \end{equation}
   which is positive on $M_0$.

\end{theorem}

We also remark that our geometric reductions and the result in \cite{SerrinZou2002} allow us to deduce a nonexistence result of positive solutions of \eqref{EP1}.

\begin{proposition} \label{propnonexis}
Assume [G1], either [G1(a)] or [G1(b)], and [G2] with $c_1=-\infty$, and  $c_2=\infty$, for an unimodular Lie group $G$. Suppose that $f$ satisfies [F0] with $f\geq 0$. Let $u$ be a nonnegative $C^2$ maximal radial solution of \eqref{EP1}. Then $u$ is constant in $M_0$. In particular, if $f>0$ in all its domain, there are no nonnegative $C^2$ maximal radial solutions of problem \eqref{EP1}. 
\end{proposition}

\begin{rmk}
Note that the above nonexistence result applies to $C^2$ radial solutions of \eqref{EP1} on the whole manifold $M$, as any such solution is also a maximal solution. In particular, when $G$ is compact (implying that $M_0$ is dense in $M$), and under the geometric conditions of Proposition \ref{propnonexis},  we can conclude that radial solutions of class $C^2(M)$, that is defined in all $M$, for which the nonlinearity satisfies [F0] and $f\geq 0$ must be constant.     
\end{rmk}

\subsubsection*{Uniqueness results}

We now present our main results concerning the uniqueness of solutions.

\begin{theorem}\label{TheoMainUniq}
Let $M$ be a $n$-dimensional Riemannian manifold satisfying condition [G1] for a 1-dimensional submanifold $\tilde{\Sigma}$, and an unimodular Lie group $G$, so that $\Sigma$ satisfies either conditions [G1(a)] or [G1(b)]
\\ \\
i) Assume [F0] and [F1]. Let $u_1, u_2$ be two maximal radial solutions of \eqref{EP1} of class $C^2$ such that: 
\begin{itemize}
       \item[{[U1]}] There exists $p\in \Sigma$ for which 
       $$ u_1(p)=u_2(p)\quad \mbox{and}\quad  \nabla u_1(p)=\nabla u_2(p).$$ 
\end{itemize}
Then $u_1 \equiv u_2$ in $M_0$.
\\ \\
ii) Assume [F0] and [F1], and that $f(\cdot,0)=0$. Let $u$ be a maximal radial solution of \eqref{EP1} of class $C^2$ such that: 
\begin{itemize}
       \item[{[U2]}] There exists $p\in \Sigma$ for which $ u(p)=0$ and $\nabla u(p)=0$. 
\end{itemize}
Then $u \equiv 0$ in $M_0$.
\\ \\
iii) Assume that $f(x,y)\in C(M\times (0,\infty))$ is radial  in the first variable, and [F1] holds with $K\subset M\times (0,\infty)$ compact. Let $u_1, u_2$ be two maximal radial solutions of \eqref{EP1} of class $C^2$ with $u_j>0$, $j=1,2$ such that [U1] holds. Then $u_1 \equiv u_2$ in $M_0$.
\end{theorem}

When $M_0$ is dense in $M$, the above theorem allows us to obtain the uniqueness of solutions of \eqref{EP1} defined on the entire manifold $M$. More precisely, as a consequence of Theorem \ref{TheoMainUniq}, we have:

\begin{corollary}\label{corolunique}
Let $M$ be a $n$-dimensional Riemannian manifold satisfying condition [G1] for a 1-dimensional submanifold $\tilde{\Sigma}$, and an unimodular Lie group $G$, so that $\Sigma$ satisfies either conditions [G1(a)] or [G1(b)]. Assume that $M_0$ is dense in $M$.

Let $u_1, u_2\in C^2(M)$ be two radial solutions of \eqref{EP1} on the whole manifold $M$. Assuming either of the conditions i) or iii) in Theorem \ref{TheoMainUniq}, it follows that $u_1\equiv u_2$ in $M$. Moreover, assuming the conditions of ii) in Theorem \ref{TheoMainUniq}, it follows that $u_1\equiv 0$ in $M$.
\end{corollary}


\subsubsection{Remarks on existence and uniqueness results}
Let us now mention some comments on our main results.

\begin{enumerate}[1.]

\item Our results in Theorem \ref{TheoMainUniq} show that assuming conditions [U1] and [U2], there must be at most one radial maximal solution of problem \eqref{EP1}. The same can be said from Corollary \ref{corolunique}. However, we are not guaranteeing the existence of radial solutions on the whole $M$. In particular, we are not contradicting results on the nonexistence of solutions, for example, see \cite{EstebanLions1982,GarafaloLanconelli1992,SerrinZou2002,WangWei2023} and references therein.\\ 
 \item  A direct computation shows that the condition    \\
 
    [U3] \emph{$u_1=u_2$ in an open neighborhood $\Omega\subset M $, which intersect $\Sigma$}  \\
    
 implies [U1] (and it also implies [U2] with $u_2=0$). Then, by replacing [U1] by [U3] (and [U2] by [U3] with $u_2=0$, resp.) in Theorem \ref{TheoMainUniq}, the results are still valid. Consequently, the uniqueness results with the condition [U3] resemble an UCP.
\item  Our uniqueness results in Theorem \ref{TheoMainUniq} and Corollary \ref{corolunique} can be extended to radial solutions $u_1$ and $u_2$ in a more general context, such as Sobolev space with regularities of order greater than or equal to 2. This is a work in progress for a future manuscript. 

\item In Subsection \ref{AddGeconstr} below, we provide further results on the existence of solutions using geometric tools, including possible extensions to the singular points of the Laplace-Beltrami operator.
\end{enumerate}
\begin{example}[Examples for the geometrical assumptions]\label{rmk:bhv of A}
    
With the following examples, we aim to show explicit situations for assumptions [G1], [G1(a)], and [G1(b)]. We refer to Section~\ref{sec:geometry} for more details on polar actions.

    \begin{enumerate}[1.]
    \item {\bf Paraboloid (and truncated paraboloid)}: Let $M$ be the paraboloid $z=x^2+y^2$, which is constructed by rotating (a $S^1$-action) the curve $\tilde{\Sigma}=\{(r,0,r^2):r\geq 0\}$. 
    Note that the action only has a fixed point (at the origin), hence the set $M_0$ is open and dense in $M$. The previous discussion makes $\Sigma=\{(r,0,r^2):r> 0\}$ an example of [G1(a)].
    We also can consider two instances of {\it truncated paraboloids}
    \begin{align*}
      M&=\{(x,y,z)\in \R^3: a\leq z=x^2+y^2\leq b\},\mbox{\ \ or\ }\\ 
      M&=\{(x,y,z)\in \R^3: z=x^2+y^2\leq b\}  
    \end{align*}

with $0<a<b$. Both cases are manifold with boundary and the respective curves for the polar action are 
\begin{align*}
\tilde{\Sigma}&=\{(r,0,r^2):a\leq r^2\leq  b,r\geq 0\},\mbox{\ \ or\ }\\ 
        \tilde{\Sigma}&=\{(r,0,r^2): 0\leq r^2\leq b,\}.        
    \end{align*}

Following the same ideas as above, we get the respective submanifold $M_0$ as 
\begin{align*}
M_0&=\{(x,y,z)\in \R^3: a< z=x^2+y^2< b\},\mbox{\ \ or\ }\\ 
M_0&=\{(x,y,z)\in \R^3: 0<z=x^2+y^2< b\},       
    \end{align*}
which are dense in their correspondingly truncated paraboloid and yield examples of [G1(b)], and the respective $\Sigma$ are given by.
\begin{align*}
\Sigma&=\{(r,0,r^2):a< r^2<  b,\},\mbox{\ \ or\ }\\ 
      \Sigma &=\{(r,0,r^2): 0< r^2< b\}.  \diamond      
    \end{align*}

    \item {\bf The Euclidean space}: $\R^n$ is an example of manifold with action as in [G1]. To see this, we can consider $\tilde{\Sigma}$ as a ray from any point $p_0$ and the action by $S^{n-1}$ rotations around $p_0$. In this case, $\Sigma=\tilde{\Sigma}\backslash \{p_0\}$ gives an example of [G1(a)], and also note that the spheres centered at $p_0$ are invariant by rotations. With a similar argument, a respective $\Sigma$ satisfying assumption [G1(b)] is related to rings (possible with boundary) centered in $p_0$ in $\R^n$ or $\R^n$ itself. $\diamond$
    
    \item {\bf The sphere}: In this example, $\tilde{\Sigma}$ can be constructed from any meridian joining the north and south poles. The same holds if we change $\tilde{\Sigma}$ by any geodesic (i.e., half of the maximal circle) joining a point and its antipodal. This is an example of [G1] (parts (a), (b)).
    Similarly to the case of a truncated paraboloid, we can consider spherical strips or spherical caps, where the respective $M_0$ submanifold are open spherical strips and open spherical caps by dropping the ''north'' pole.
    $\diamond$

    \item {\bf Surface of revolution}. This is somehow a paradigmatic example and works similarly to the paraboloid. An advantage of this surface is that it helps us to give us a function $A$ such that [A1] and [A2] hold. For example, consider $M$ be the revolution surface of the curve $$\Sigma=\{(r,0,R(r)):r>0\}$$
    with $R:[0,\infty)\to \R$ continuous and smooth on $(0,\infty)$ such that
    $$ R(r)=r^{1/2} \mbox{\ in\ } [0,3/4), $$ 
        for which $A^{-1}$ is integrable in a neighborhood of the origin. $\diamond$ 
        
\item {\bf Non-compact case}
We will study $\mathbb{R}^3$ regarded as a polar action. For this, consider $\tilde{\Sigma}$ as the $z$-axis, with the orbits given by parallel planes of constant $z$ (i.e., understanding the action as an $\mathbb{R}^2$ translation,  $G=\mathbb{R}^2$ is an unimodular Lie group). Considering the Euclidean metric in $\mathbb{R}^3$, the orbits are non-compact with infinite {\it volume}. However, we can represent $\mathbb{R}^3=\mathbb{R}^2\times \mathbb{R}$ with the so-called {\it warped metric}, for which we consider the usual metric for the $z$-component while  for $\mathbb{R}^2$ we define a conformal metric of the form
$F\langle  \cdot,\cdot \rangle_{\mathbb{R}^2}$
with $F(x,y,z)=e^{-(x^2+y^2)}$. This conformal factor $F$ is introduced to obtain a finite volume of the orbits, which shows that $\mathbb{R}^3$ satisfies [G1(a)] with non-compact unimodular Lie group action.$\diamond$ 
    \end{enumerate}
To close the examples of geometrical assumptions, we just comment that all the previous situations also satisfy [G2] just by reparametrization of the curves $\Sigma$ at each case. 
\end{example}

We now continue by showing examples that satisfy the analytical assumption in the previous section.

\begin{example} Let us exhibit some functions $f$ that satisfy the conditions in Theorem~\ref{TheoMainExis} and Theorem~\ref{TheoMainUniq}.
\begin{enumerate}[1.]
    \item The local existence result in Theorem \ref{TheoMainExis} i) is quite general, thus one has many examples of nonlinear terms $f$ satisfying [F0] and [F1]. In particular, our results apply to the geometric  version of the equation associated with solitary wave solutions of the nonlinear Schr\"odinger equation, i.e.,
\begin{equation*}
   -\Delta_Mu+u-|u|^{p-1}u=0, 
\end{equation*}
where $p> 1$, thus $f(x,y)=y-|y|^{p-1}y$ in \eqref{EP1}. Next, writing $f(x,y)=b(x)|y|$ for some appropriate radial function $b\in C(M)$ such that [F2] holds, we get an example of Theorem \ref{TheoMainExis} ii). In the case of Theorem \ref{TheoMainExis} iii), consider the problem \eqref{NegativeSigmaproblem}, i.e., $f(x,y)=b(x) y^{-\sigma }$, with $0<\sigma< 1$, and take an appropriated function $b\in C(M)$ such that [F3], [F4] and [F5] are valid (assuming that $A$ satisfies [A1]). Under condition [A2], as an example of  Theorem \ref{TheoMainExis} iv) consider \eqref{EP1} with $f(x,y)=b(x) y^{\sigma }$, $\sigma>1$, where the radial function $b$ is such that [F6] is valid. For some additional examples and extensions of Theorem \ref{TheoMainExis}, see Remark \ref{otherexam} below $\diamond$

\item Note that uniqueness results apply to more general examples than those concerning the existence of solutions. Thus, besides examples above, some additional cases of nonlinearities $f: M\times \mathbb{R}\rightarrow \mathbb{R}$ that satisfy the hypotheses of Theorem \ref{TheoMainUniq} i) and ii) include $f(x,y)=b(x)y^{\gamma}$, $f(x,y)=b(x)|y|^{\gamma}$ with $\gamma\geq 1$, and finite linear combination of such examples. Also, the function $f(x,y)=b(x)e^{y^{\gamma}}$, $\gamma>1$ (with $b\neq 0$) satisfies the hypothesis of Theorem \ref{TheoMainUniq} i), but not those in part ii).  Examples that satisfy conditions of Theorem \ref{TheoMainUniq} iii) include:  $f(x,y)=b(x)e^{y^{-\sigma}}$, $\sigma>0$, $f(x,u)=b(x)\log u$ and finite linear combination of such functions.   $\diamond$
\end{enumerate}    
\end{example}



\section{Geometrical setting}\label{sec:geometry}

In this section, we present the basics of Riemannian geometry needed to set and study problem \eqref{EP1}. In particular, we include Corollary \ref{cor:reduction to ODE}, which allows us to reduce the dimension of \eqref{EP1}.

\subsection{Basics on Riemann geometry}
We consider a Riemannian manifold $(M,g)$ where  $M$ is a smooth manifold and $g$ is a Riemannian metric tensor. For the convenience of the reader, we refer to \cite{Jo} for the standard definitions and properties of the length of a curve, the Levi‑Civita connection, geodesics, the Riemann curvature tensor, the Ricci curvature, and other related concepts. Here, we will focus on specific classes of examples and symmetries and isometries of the structure.

We begin with some classes of Riemannian manifolds that are relevant to the main results of this paper:

\begin{example}\label{ex:initialex}
{\it Euclidean space:} The simplest case is given by any open set of $\mathbb{R}^n$ with its usual topology and differential structure, and the Riemannian metric comes from the usual inner product and Euclidean norm.

{\bf The sphere:}  The smooth structure of $S^2$ can be constructed in several ways (as a level set of the smooth function $f(x)=\|x\|^2$, or stereographic coordinates, spherical coordinates). However, for our purposes, we present the {\bf radial circles} construction (similarly as that in \cite{BecerraGalvisMartinez2019}). Let $C_z$ be the centered circle at some point in $[-r_0,r_0]$ in the $z$-axis. It provides us a parametrization of $S^2$ via the formula  $$\gamma(z,t)=(\sqrt{r_0^2-z^2}\cos(t), \sqrt{r_0^2-z^2}\sin(t),z)$$ for a fixed value of $r_0$. 
We consider the tangent space $T_pS^2$ at some point $p=(0,0,r_0)\in S^2$ and we can project each circle defined on these polar coordinates to $S^2$ such that the image of those circles coincides with the circles $C_z$, and the radio $r$ projects on a geodesic transversal to the circles $C_z$. The same construction can be extended to $S^n$, $n$-dimensional spheres. Finally, the Riemannian metric for this manifold is the induced one by the usual inner product on $\mathbb{R} ^{n+1}$ and satisfies that the curvature is constant 1. 

{\bf Closed surfaces of evolution:} Consider now a plane curve $\gamma(t)=(x(t),z(t))$ defined on an interval $I=[a,b]$ with non-negative components and parametrized by arclength. If in addition, we assume that $x(a)=x(b)=0$, $x'(a)x'(b)\neq 0$, and $z'(a)=z'(b)=0$, then we can construct a simply connected surface of revolution $S$ parametrized by $(x(t)\cos(\theta),x(t)\sin(\theta),z(t))$. The surface $S$ has a Riemannian structure coming from the Euclidean structure on $\R^3$. 

{\bf Projective spaces:} In the Euclidean space $\mathbb{R}^{n+1}\backslash \{0\}$ (without the origin) we define $P\mathbb{R}^n$  the set of the lines through the origin. Indeed, we can equip $\mathbb{R}^{n+1}\backslash \{0\}$ with the quotient topology by the equivalence relation given by $v\sim u$ if and only if $u=rv$ for non-vanishing scalar $r$. This topology makes $P\mathbb{R}^{n}$ into a smooth manifold, which is called real projective space. The same construction can be done for the complex space $\mathbb{C}^{n+1}$ and the $n$-quaternionic $\mathbb{H}^{n+1}$, making the complex and quaternionic projective spaces, respectively. In these three cases, we can define the smooth projection $\pi:\mathbb{K}^{n+1}\to P\mathbb{K}^n$ (where $\mathbb{K}$ is any one of the previous algebras). The Riemannian geometry of these structures allows us to obtain a positive constant sectional curvature. 

{\bf Hyperbolic spaces:} Again, if we denote $\mathbb{K}^n$ one of the spaces $\mathbb{R}^{n+1}$, $\mathbb{C}^{n+1}$ or $\mathbb{H}^{n+1}$, and endowed them with $\Psi$ a $(n,1)$-signed hermitian form, we have a non-vanishing space $V_-$ as negative vector with respect to $\psi$. We define the space $H\mathbb{K}^n:=\pi(V_-)$, and it is endowed with a smooth structure. These spaces are called real, complex, and quaternionic hyperbolic spaces. Each of the previous spaces has a Bergman metric, that allows us to obtain examples of negative constant sectional curvature \cite{ChenGreenberg1974}. $\diamond$
\end{example}

\subsection*{Symmetries and isometries} Among the morphisms of smooth manifolds, i.e., smooth functions, we should consider the {\it isometries}. The isometries are diffeomorphisms that preserve the metric structure, i.e., for $(M,g_M)$ and $(N,g_N)$, an isometry is a diffeomorphism $\phi:M\to N$ so that
$$g_M|_p(X,Y)=g_N|_{\phi(p)}(d_p\phi X, d_p\phi Y)$$ 
for any $p\in M$ and any tangent vector $X,Y$ on $T_pM$. Note that the composition of functions satisfies $d_x(f\circ g)=d_{g(x)}f\circ d_xg$ (where the composition $f\circ g$ is well defined), then isometries of a fixed manifold inherit a group structure. 

A  Lie group\footnote{A Lie group is a manifold with a group structure, where the product is a smooth function and the inverse map is a diffeomorphism.} acts by ismotries on a manifold $M$ if there is a smooth application $\phi:G\times M\rightarrow M$ sucht that $\phi(\alpha,\cdot):M\to M$ is an isometry  and it is an action on $M$, i.e.

$$ \phi(e,x)=x,\qquad  \phi(\alpha\beta,x)=\phi(\alpha,\phi(\beta,x)) $$
for all $\alpha,\beta\in G$ and for all $x\in M$, where $e\in G$ is the identity element. With the previous notation, we have defined the polar action as in the condition [G1] in Subsection~\ref{sec:geom assumption}.

An exceptional notion associated with isometries is a certain kind of symmetric spaces.

\begin{example}[Two points homogeneous spaces, cf. \cite{HelgasonII}]\label{two-point homogeneous space}
A Riemann manifold $(M,g)$ is {\it two-point homogeneous space}, if the group $I(M)$ of isometries acts transitively on the space of an equidistant pair of points. The definition means that for any $p_1,p_2,p_1',p_2'\in M$ so that $d_g(p_1,p_2)=d_g(p'_1,p'_2)$ then there exists $\varphi\in I(M)$ so that $\varphi(p_i)=p_i'$.  As a direct consequence of the definition, we have that the function  $A_p(r)$ defined as the Riemann measure of the geodesic sphere $S_r(P)$ with center at $p\in M$ is independent of the point $p$, hence $A$ is globally defined on $M$. $\diamond$
\end{example}

\begin{rmk}\label{rmk_two-point homogeneous space}
\begin{enumerate}
    \item The isometry group of a Riemannian manifold $M$ of dimension $n$ has dimension at most $n(n+1)/2$. If the manifold $M$ is simply connected, and the group reaches the maximal dimension, then $M$ is isometric to the sphere, Euclidean space, or hyperbolic space.  When $M$ is not simply connected and the isometry group has dimension $n(n+1)/2$, then $M$ is a projective space. For more details on this topic, we refer to \cite{Kobayashi1972}.

    \item  For the complete classification of the symmetric spaces (of arbitrary rank), we refer to \cite[Table V, Chapter X]{Hel2}, and we just give the list of the two-point homogeneous space (see also, \cite[Section 1.1]{Shchepetilov}) that contrast with previois comment and Example~4 below. 
\begin{itemize}
\item The Euclidean space $\R^n$ and the spheres $S^n$ for $n\geq 1$.
\item Real, complex and quaternion projective spaces $P\R^n, P\C^n,$ and $P\mathbb{H}^n $ for $n\geq 2$.
\item Real, complex and quaternion hyperbolic space $H\R^n, H\C^n,$ and $H\mathbb{H}^n $ for $n\geq 2$.
\item Cayley projective $P\C_a^2$ and hyperbolic spaces $H\C_a^2$.
\end{itemize}
Also, some special cases of closed surfaces of revolution, classified by the Gaussian curvature, are isometric with two-point homogeneous space. 
\end{enumerate}

\end{rmk}

With the notion of actions by isometries, we proceed to revisit some of the previous examples. 

\begin{example}[Examples revisited]\label{ex:polaract}
Here we recover the examples \ref{ex:initialex} regarding polar coordinates.
\begin{enumerate}[1.]
\item For $M=\mathbb{R}^n$, there is an action of the group $SO(n)$. For this action the submanifold $\tilde{\Sigma}$ can be chosen as an infinite line from the origin in $\R^n$. $\diamond$
\item  For any sphere $S^n$, the subgroup $G$ of $SO(n+1)$ defined by the rotations about the $x_{n+1}$-axis acts on $S^{n}$ by considering it as a submanifold of the space $\R^{n+1}$ (that is, $G=SO(n-1)$). The submanifold $\tilde{\Sigma}$ can be chosen as a geodesic line joining the points $(0,\dots ,0, 1)$ and $(0,\dots ,0, -1)$ of $S^n$. $\diamond$
\item For the surfaces of revolution, the group $S^1$ acts by rotations on the surface, where $\tilde{\Sigma}$ is image of the curve $\gamma(t)=(x(t),0,z(t))$. Moreover, note that this holds for any surface of revolution, not only the closed ones. $\diamond$
\item In general, in any two-point homogeneous space, there always exists a transversal one-dimensional submanifold $\tilde{\Sigma}$. Indeed, consider the rotations on the normal coordinates and use that the exponential map covers the manifold except for fixed points (this holds because geodesic is globally defined (cf., Example~\ref{two-point homogeneous space}), hence the image by $exp$ of any line is the desired transversal submanifold $\tilde{\Sigma}$. The previous result is also proved in  \cite[Theorem~4.6 and 4.10]{Heintze} for the compact case. $\diamond$

\end{enumerate}
\end{example}

All these considerations also hold for two-point homogeneous space. In this way, the polar geodesics coordinates on these spaces can be understood as the ones given by the orbits of an action and the transversal 1-dimensional submanifold.


\subsection{The Laplace--Beltrami operator for radial functions}

Here we use the splitting granted by polar action to describe the Laplace-Beltrami operator on manifolds. To exemplify such a procedure, we recall a well-known situation,  the Laplace operator in  the case of Cartesian coordinates $(\R ^2,x,y)$ and in polar coordinates $(\R^2,r,\theta)$: 
$$
\Delta u=\dfrac{\partial^2u}{\partial x^2}+\dfrac{\partial^2u}{\partial y^2}=\dfrac{\partial^2u}{\partial r^2}+ \dfrac{1}{r}\dfrac{\partial u}{\partial r}+\dfrac{1}{r^2}\dfrac{\partial^2 u}{\partial ^2 \theta},
$$
and under the assumption that $u$ is a radial function, we have the following:
$$\Delta u=u''+\frac{1}{r}u'.$$
It is important to highlight that this only works out of the origin because at $r=0$, the Laplace--Beltrami operator has a singularity.\\ 
To implement the previous idea in more general settings, we begin with the description of the Laplace-Beltrami operator for Riemann manifolds $M$. If we denote by $\bar{g}=\det(g_{ij})$ for some coordinate chart, the Laplace-Beltrami operator acting on a smooth function
$u:M\to \mathbb{R}$ is defined as  
$$\Delta_Mu=\frac{1}{\sqrt{\bar{g}}}\sum_{k}\frac{\partial}{\partial x_k}\left( \sum_{i} g^{ik}\sqrt{\bar{g}}\frac{\partial}{\partial x_i}u \right)$$ where $(g^{ij}):=(g_{ij})^{-1}$ and $\bar{g}=$det$(g_{ij})$. For the coordinate-free expression of $\Delta_Mu$ and other properties of this operator, we refer to \cite[Section 2.1]{Jo}.

The main advantage is that for manifolds that admit polar actions, the strategy presented in \cite{HelgasonIII} gives a general decomposition of the Laplace-Beltrami operator via its radial part, i.e., the function $u$ does not depend on the orbit direction. In \cite[Section 3.3]{BecerraGalvisMartinez2019},  there is an extended summary of the construction in \cite{HelgasonIII}, and in this manuscript, we just mention the main results needed to simplify the Laplace--Beltrami operator for radial functions.

\begin{theorem}[\cite{HelgasonIII}, Proposition 2.1]\label{teoradialpart}
Let us suppose that $\Sigma$ is transversal to the action of $G$. Let $D$ be a differential operator on $M$. There is a unique operator $\Delta(D)$ on $\Sigma$ such that 
$D(u)|_{\Sigma}=\Delta(D)(u|_{\Sigma})$, for each $G$-invariant function $u$ defined on an open subset of $M$. 
\end{theorem}
The operator $\Delta(D)$ is called the \emph{radial part} of $D$. If we denote $L$ as the Laplace-Beltrami operator\footnote{When we need to specify the operator on a particular manifold, we denote it with a subindex, denoting the manifold where is computed.}, we have the particular situation as follows: \\
\begin{theorem}[\cite{HelgasonIII}, Theorem 2.11]\label{t8} For any submanifold $\Sigma\subset M$, transversal to the action an unimodular compact Lie group $G$ without fixed points, we have
\begin{equation}\label{eq:radial of LB}
\Delta(L_M)=\frac{1}{\sqrt{A(r)}}L_{\Sigma}(\circ\sqrt{A(r)})-\frac{1}{\sqrt{A(r)}}L_{\Sigma}(\sqrt{A(r)}),
\end{equation}
where $A(r)$ denotes the 
Riemannian measure of the orbit $G\cdot r$ for each $r\in \Sigma$, and $\circ$ denotes the composition of the operator $L_{\Sigma}$ and the operator multiplication by $\sqrt{A(r)}$. 
\end{theorem}
From the previous theorem, it is worth making a couple of remarks:
\begin{itemize}

\item An interesting observation is that radial solutions of problem \eqref{EP1} depend on the geometry of the orbits rather than on the whole geometry. We have examples of various types of geometries (e.g., two-point homogeneous spaces with different types of curvatures) for which existence and uniqueness results hold.

    \item We observe that, in the case of exponential or polar coordinates, the orbit $G\cdot r$ coincides with the geodesic ``sphere'' of geodesic radius $r$.
    \item Note that the fixed points of the action introduce singularities in the Laplace--Beltrami operator, as it was mentioned in Remark \ref{rmk:bhv of A}, where different examples of $A$ are provided.
\end{itemize}

Applying the previous theorem to (local) radial functions,  where $x\mapsto f(x,\cdot)$ is also radial, a  straightforward computation on problem \eqref{EP1} yields
\begin{equation}\label{eqdimg}
\Delta_\Sigma u+2\dfrac{\nabla_\Sigma \sqrt{A(r)}\cdot \nabla_\Sigma u}{\sqrt{A(r)}}+f(r,u(r))=0.
\end{equation}
Equivalently, it can be written as
\begin{equation}
\Delta_\Sigma u+{\nabla_\Sigma \ln{A(r)}\cdot 
\nabla_\Sigma u}+f(r,u(r))=0
\end{equation}
on non-fixed points. Nevertheless, we also consider another simplification:  We assume now the existence of a one-dimensional transverse submanifold $\Sigma$. Same computation as before, but for a one-variable function $u$ yields
\begin{align*}L_M(u)&=\frac{2\nabla(\sqrt{A(r)})\cdot\nabla(u)}{\sqrt{A(r)}}+L_{\Sigma}(u)\\
&=\frac{(A(r))'u'}{A(r)}+u''=(\ln(A))'u'+u''.
\end{align*}

Summarizing, we have deduced the following key result.

\begin{corollary}\label{cor:reduction to ODE} Under the assumptions of Theorem  \ref{t8}, and given $\Sigma $ with dimension one, we have
$
L_M(u)=(\ln(A))'u'+u''$. Consequently, if $x\mapsto f(x,\cdot)$ is radial, problem \eqref{EP1} for radial functions is equivalent to
\begin{equation}\label{eq:EDO from EP1}
u''+(\ln(A))'u'+f(r,u(r))=0    
\end{equation}
defined on $\Sigma$.
\end{corollary}

\begin{rmk}

Recall that $u'$ means the derivative of a curve on a Euclidean space, but in general, $\Sigma$ is not a line (or open set of a line), and the notion of $u'$ as a limit of the usual differential quotient makes no sense. However, we still can give enough sense in the context of smooth manifolds as follows: As the function $u:\Sigma\to \mathbb{R}$ is locally given by $u:(-\epsilon,\epsilon)\to \mathbb{R}$, then we have
$$d_tu:T_t(-\epsilon,\epsilon)\to T_{u(t)}\mathbb{R},$$
and recall that the vector space $T_t(-\epsilon,\epsilon)$ is generated by a tangent vector $\partial_t$ from the local chart $((-\epsilon,\epsilon),t)$, then $u'=d_tu(\partial_t)\in T_{u(t)}\mathbb{R}=\mathbb{R}$.

\end{rmk}

\section{Proof of existence results}\label{sec:existe}

In this section, we apply the different assumptions on $\tilde{\Sigma}$, $\Sigma$, $A$, and the nonlinear term $f$ to deduce the existence of solutions for problem \eqref{EP1}. In addition, we exhibit a result of non-existence of positive solutions in Proposition \ref{propnonexis}. The key idea to finding radial solutions is to use conditions [G1(a)] and [G1(b)] in the problem \eqref{EP1} to arrive at the equivalent problem
\begin{equation}\label{EP2}
\begin{aligned}
   &(A(r)u'(r))'+A(r)f(r,u(r)) =0,
 \end{aligned}   
\end{equation}
and apply different results from the theory of ODEs.

Let us begin with the deduction of Theorem \ref{TheoMainExis}, which we will divide into parts i)-v).

\begin{proof}[Proof of Theorem \ref{TheoMainExis} part i)]
As commented above, it is enough to study the  ODE in problem \eqref{EP2}, which is well-posed on $\Sigma$. Thus, we obtain the local existence of solutions from a standard application of Picard-Lindel\"of theorem. Finally, extending $u$ by $G$-symmetries, i.e., $u(g\cdot p):=u(p)$, we get a local radial solution of \eqref{EP1}.
\end{proof}

Before continuing with the proof of Theorem \ref{TheoMainExis}, we apply condition [G2] to further simplify equation \eqref{EP2}. Since the function $A(r)$ is nonnegative, we define $s=J(r)=\int_{r_0}^r A(t)^{-1}\, dt$ with $r_0 \in  (0,\mbox{lenght}(\Sigma))$, then 
\begin{equation*}
  \begin{aligned}
     \frac{dr}{ds}=A(r), \quad \text{and } \quad \frac{dz}{ds}=u'(r)A(r), 
  \end{aligned}  
\end{equation*}
where
\begin{equation*}
    z(s)=u(J^{-1}(s)).
\end{equation*}
It follow that equation \eqref{EP2} takes the form 
\begin{equation}\label{ODEeq}
    \begin{aligned}
    z''(s)+A(r(s))^2f(r(s),z(s))=0.
\end{aligned}
\end{equation}

\begin{proof}[Proof of Theorem \ref{TheoMainExis} part ii)]
We first note that since $A(r(s))^2$ is a function with continuous derivative, setting $\widetilde{f}(s,y)=A(r(s))^2f(r(s),y)$ with $s\in (c_1,c_2)$, $y\in \mathbb{R}$, it follows that $\widetilde{f}$ also satisfies [F1] and [F2] but on the set $(s,y)\in (c_1,c_2)\times \mathbb{R}$. Consequently, we can apply classical global existence theorems for ODE (see for example \cite{Sideris2013,Gerald2012}) to deduce that for any initial conditions in $(c_1,c_2)$, there exists a solution $z(s)$ of class $C^2(c_1,c_2)$ of problem \eqref{ODEeq}. Now, extending $u(r)$ by $G$-symmetries to whole $M_0$, $u$ gives a maximal radial solution of \eqref{EP1}.
\end{proof}

\begin{proof}[Proof of Theorem \ref{TheoMainExis} part iii)]
As before, we can reduce the existence problem to that of \eqref{ODEeq}, where we assume $-\infty<c_1<c_2<\infty$. If we make the change of variables $s=(c_2-c_1)t+c_1$ with $t\in (0,1)$, and consider $w(t)=z(s(t))$ on\eqref{ODEeq}, we have 
\begin{equation}
  \frac{d^2 w}{dt^2}+\widetilde{A}(t)^2\widetilde{f}(t,w(t))=0,  
\end{equation}
where $\widetilde{A}(t)=(c_2-c_1)A(r(s(t)))$, and $\widetilde{f}(t,w(t))=f(r(s(t)),w(t))$. With this change of variables, the conditions [A1], [F3], [F4] and [F5] in terms of $\widetilde{A}$ and $\widetilde{f}$ now read as follows:
\begin{itemize}
    \item[{[A'1]}] $\widetilde{A}\in C([0,1])$ with $\widetilde{A}>0$ in $(0,1)$.
        \item[{[F'3]}] $\widetilde{f}: (0,1)\times (0,\infty)\rightarrow (0,\infty)$ is continuous, $y \mapsto \widetilde{f}(t,y)$ is decreasing for each $t\in (0,1)$, and $t \mapsto \widetilde{f}(t,y)$ is integrable for each $y$.  
     \item[{[F'4]}] We have $$\lim_{y\to 0^{+}}\widetilde{f}(t,y)=\infty \, \, \text{and } \, \, \lim_{y \to \infty}\widetilde{f}(t,y)=0,$$ 
     uniformly on compact subsets of $(0,1)$.
     \item[{[F'5]}] It follows
     \begin{equation*}
         0<\int_{0}^{1} \widetilde{A}(t)^2\widetilde{f}\big(t,g_{\theta}(t)\big)\, dt, 
     \end{equation*}
     for all $\theta>0$, where $g_\theta (t):= \theta\min\{t, (1-t)\}$.
\end{itemize}
Given $\alpha,\gamma, \beta,\delta \geq 0$ such that $\gamma \beta+\alpha \gamma+\alpha\delta>0$, we define the boundary value problem
\begin{equation}\label{BBP1}
\left\{\begin{aligned}
&\frac{d^2 w}{dt^2}+\widetilde{A}(t)^2\widetilde{f}(t,w(t))=0,\quad t\in (0,1), \\
&\alpha w(0)-\beta w'(0)=0, \\
&\gamma w(1)+\delta w'(1)=0.
\end{aligned}\right.    
\end{equation}
It follows from [A'1], [F'3], [F'4] and [F'5] that we can apply the results in \cite[Theorem 2.2]{GaticaOlikerWaltman1989} to obtain that \eqref{BBP1} has a positive solution $w$. Consequently, by reversing the change of variables, we obtain a maximal solution $u$ of \eqref{EP1} with the desired properties.
\end{proof}

 We continue with the study of the case $\Sigma$ being parameterized by the interval $(a,\infty)$. 

\begin{proof}[Proof of Theorem \ref{TheoMainExis} iv)]
We first consider the change of variables $r=t+a$, mapping $t\in [0,\infty)$ into $s\in [a,\infty)$. Next, we set $w(t)=u(r(t))$, $\widetilde{A}(t)=A(r(t))$, $\widetilde{f}:(0,\infty)\times (0,\infty) \times (0,\infty)\rightarrow \mathbb{R}$ be given by $\widetilde{f}(t,y,z)=f(r(t),y)$, and $\widetilde{\rho}(t)=\rho(r(t))=\int_0^t \widetilde{A}(t')^{-1}\, dt'$. Using equation \eqref{EP2}, we look for solutions of problem
\begin{equation}\label{CPB2}
\left\{\begin{aligned}
&\frac{1}{\widetilde{A}}\Big(\widetilde{A} w'\Big)'+\widetilde{f}(t,w,\widetilde{A}w')=0,\quad t\in (0,\infty), \\
&w>0, \quad  \text{ in } (0,\infty), \\
&\lim_{t\to 0^{+}} w(t)=0.
\end{aligned}\right.    
\end{equation}
Conditions [A2] and [F6] imply that $\widetilde{A}$ and $\widetilde{f}$ satisfy the hypothesis in \cite[Theorem 1]{MaagliMasmoudi2001} (see also \cite{BacharMaagli2014,Zhao1994}). Hence, for such a result, there exists a solution $w$ of \eqref{CPB2} in the class $C([0,\infty))\cap C^1((0,\infty))$, from which $u(r)=w(r-a)$ and $G$ action yield a $C^1$ maximal radial solution of  \eqref{EP1}. Moreover, by the results in \cite[Theorem 1]{MaagliMasmoudi2001}, there exists some $c>0$ such that
\begin{equation*}
        u(r)=c\rho(r)+\int_a^{\infty} A(t)\rho(g_1(r,t))\big(1-\frac{\rho(g_2(r,t))}{\rho(\infty)}\big) f(t,u(t))\, dt,
    \end{equation*}
and
 \begin{equation*}
        \lim_{r\to a^{+}} u(r)=0, \qquad \lim_{r\to \infty}\frac{u(r)}{\rho(r)}=c,
\end{equation*}
where $\rho(\infty)=\lim_{r\to \infty}\rho(r)$, $g_1(r,t)=\min\{r,t\}$, and  $g_2(r,t)=\max\{r,t\}$. As before, by $G$-extension of $u$, we get the desired result.
\end{proof}
Using the results in \cite{Bachar2005}, we can find solutions of \eqref{EP1} in the particular case $f(x,y)=b(x)y^{-\sigma}$, $\sigma>0$. 

\begin{proof}[Proof of Theorem \ref{TheoMainExis} v)] We consider again the linear change of variable $r=t+a$, mapping $[0,\infty)$ into $[a,\infty)$. Changing to polar coordinates, we set $w(t)=u(r(t))$, $\widetilde{A}(t)=A(r(t))$, $\widetilde{\rho}(t)=\rho(r(t))$, with $\rho(r)$ as in \eqref{rhodef},  and we define $\varphi(t,y)=b(r(t))y^{-\sigma}$. Hence, from \eqref{EP2}, we consider problem
\begin{equation}\label{CPB3}
\left\{\begin{aligned}
&\frac{1}{\widetilde{A}}\Big(\widetilde{A} w'\Big)'+\varphi(t,w(t))=0,\quad t\in (0,\infty), \\
&w>0, \quad  \text{ in } (0,\infty), \\
&\lim_{t\to 0^{+}} w(t)=0.
\end{aligned}\right.    
\end{equation}
Our assumptions in v) Theorem \ref{TheoMainExis}  imply that $\varphi$ satisfies conditions (H1), (H2) and (H3) in \cite[Theorem 2]{Bachar2005}, from which we obtain a solution  $w\in C([0,\infty))\cap  C^1((0,\infty))$ of \eqref{CPB3}. Consequently, setting $u(r)=w(r-a)$, we get a $C^2$ maximal radial solution \eqref{EP1}.  In addition, reversing variables, the results in \cite{Bachar2005} imply that in polar coordinates
    \begin{equation*}
    u(r)=\int_a^{\infty}A(t)\rho(\min\{r, t\})b(t)u(t)^{-\sigma}\, dt,
    \end{equation*}
$u$ is positive, and
\begin{equation*}
    \lim_{r\to \infty} \frac{u(r)}{\rho(r)}=0.
\end{equation*}

\end{proof}

\begin{rmk}\label{otherexam}
To show the existence of maximal solutions in the case $f(x,u)=b(x)u^{\sigma}$, $\sigma>0$ in \eqref{EP1} and $b$ being radial continuous function, we can use similar ideas to those given in the proof of Theorem \ref{TheoMainExis}, together with the result in \cite[Theorem 1.5]{BacharMaagli2014}. Analogously, we can consider the results in \cite{BacharMaagli2016} to show existence of solutions of \eqref{EP1} in the combined case $f(x,u)=b_1(x)u^{\sigma_1}+b_2(x)u^{\sigma_2}$. For brevity, we will not state these results in this manuscript.      
\end{rmk}

We conclude this section with the nonexistence results contained in Proposition \ref{propnonexis}.

\begin{proof}[Proof of Proposition \ref{propnonexis}]
Let $u$ be a nonnegative $C^2$ maximal radial solution of \eqref{EP1}. By [G2], we have that for some $r_0$, the change of variables $s=J(r)=\int_{r_0}^r A(t)^{-1}\, dt$ and the fact that $u$ is a maximal solutions of \eqref{EP1} imply that $z(s)=u(J^{-1}(s))$ is a nonnegative solution of equation \eqref{ODEeq} in all $\mathbb{R}$ (here we use the assumption $c_1=-\infty$, $c_2=\infty$). Since $A(r(s))^2\geq 0$, and $f\geq 0$, it follows
\begin{equation*}
   z''(s)=-A(r(s))^2f(r(s),z(s))\leq 0,
\end{equation*}
for all $s\in \mathbb{R}$. Hence, setting $n=1$, and $m=2$ in \cite[Theorem II (a)]{SerrinZou2002}, we deduce that $z$ is constant, and in turn, $u$ is constant as desired. 
\end{proof}

 \subsection{Some additional geometric constructions}\label{AddGeconstr}

The previous section deals with analytical assumptions for the existence and extension of radial solutions to the problem \eqref{EP1}. As we have a geometrical setting, we want to show that we can also use pure geometric considerations to obtain some existence results.

We begin with the existence results from the geometric assumption on $\Sigma$. More precisely, consider the canonical change of variable $\phi_1=u, \phi_2=u'$ which, together with \eqref{EP2} yield the following system of ODEs
\begin{equation}\label{eq:SystemODE}
    \begin{cases}
\phi_1'(r)=\phi_2(r),& \\
A(r)\phi_2'(r)=-A'(r)\phi_2-A(r)f(r,\phi_1(r)).& 
\end{cases}
\end{equation}
We have that \eqref{eq:SystemODE} is associated with a vector field over $\Sigma \times \Sigma$, where $ \Sigma$ is a submanifold as in either [G1(a)] or [G1(b)]. The idea is to show that the vector field associated with \eqref{eq:SystemODE} is complete, i.e., it is defined in all $\Sigma \times \Sigma$. Following \cite[Theorem 2.4]{GliklikhMorozova2004}, it is enough to guarantee the existence of a proper function $\varphi:\R\times \Sigma \times \Sigma\to \mathbb{R}$, $\varphi=\varphi(t,x_1,x_2)$, and a constant $C$ so that 

\begin{equation}\label{cond:v.f.}
    \Big|\frac{\partial \varphi}{\partial t}\Big|\leq C, \quad \Big|x_2\frac{\partial \varphi}{\partial x_1}\Big|\leq C\quad \mbox{ and } \quad \Big|(x_2A'+Af)\frac{\partial \varphi}{\partial x_2}\Big|\leq C A(r),
\end{equation}
where $(x_1,x_2)$ is the local chart in $\Sigma\times \Sigma$. 

\begin{proposition}
Let $M$ be a $n$-dimensional Riemannian manifold satisfying condition [G1] for a 1-dimensional submanifold $\tilde{\Sigma}$ and an unimodular Lie group $G$, so that $\Sigma$ satisfies either conditions [G1(a)] or [G1(b)]. If in addition there exists $\varphi:\R \times \Sigma\to \R$   satisfying conditions in \eqref{cond:v.f.}, then there existsa  radial solution of problem \eqref{EP1} in $M_0$. 
\end{proposition}

\begin{proof}
    Note that from conditions on $\varphi$, we verify the hypothesis in \cite[Theorem 2.4]{GliklikhMorozova2004}, which guarantees that the vector field associated with the problem \eqref{cond:v.f.} is complete. Hence, for any initial condition,  we get a solution of \eqref{eq:SystemODE} on the whole $\Sigma$. Thus, setting $u=\phi_1$, and extending it by symmetries to $M_0$, we obtain a solution of \eqref{EP1}. 
\end{proof}

It is worth recalling that solutions here are in the sense of in Definition \ref{def:solution}, where we cannot guarantee a solution on fixed points of the action as the Laplace--Beltrami operator has singularities on those points. However, in some situations, the solution may extend to fixed points of the action in a smooth way. Let us look at this in more detail.

Firstly, we should note that manifolds with polar action may have different transverse submanifolds defining the manifold $M$ by the action, e.g., $\tilde{\Sigma}_g$ being a translation of the original $g\cdot \tilde{\Sigma}$, for fixed $g\in G$. In this case, both transversal manifolds, $\tilde{\Sigma}$ and $\tilde{\Sigma}_g$, only intersect on fixed points of the action (See Figure \ref{SigmaSphere}, left figure for the case of the sphere). However, this is not the unique situation where there exist different $\Sigma_0$ and $\Sigma_1$ transverse submanifolds for a given polar action on $M$. For example, consider $\Sigma_0$ in the sphere as a meridian from the south pole to the north pole, and $\Sigma_1$ from a point on the equator to its antipodal, passing through one point of $\Sigma_0$ (see Figure \ref{SigmaSphere}, center figure). Both submanifolds are transverse to a suitable $S^1$-rotation leaving two fixed points. Both submanifolds have non-empty intersections; at least they have one point in common. Nevertheless, we could also have the intersection as the north hemisphere of $\Sigma_0$ (see Figure \ref{SigmaSphere} right figure). 

\begin{figure}[ht]
\centering
\includegraphics[scale=0.2]{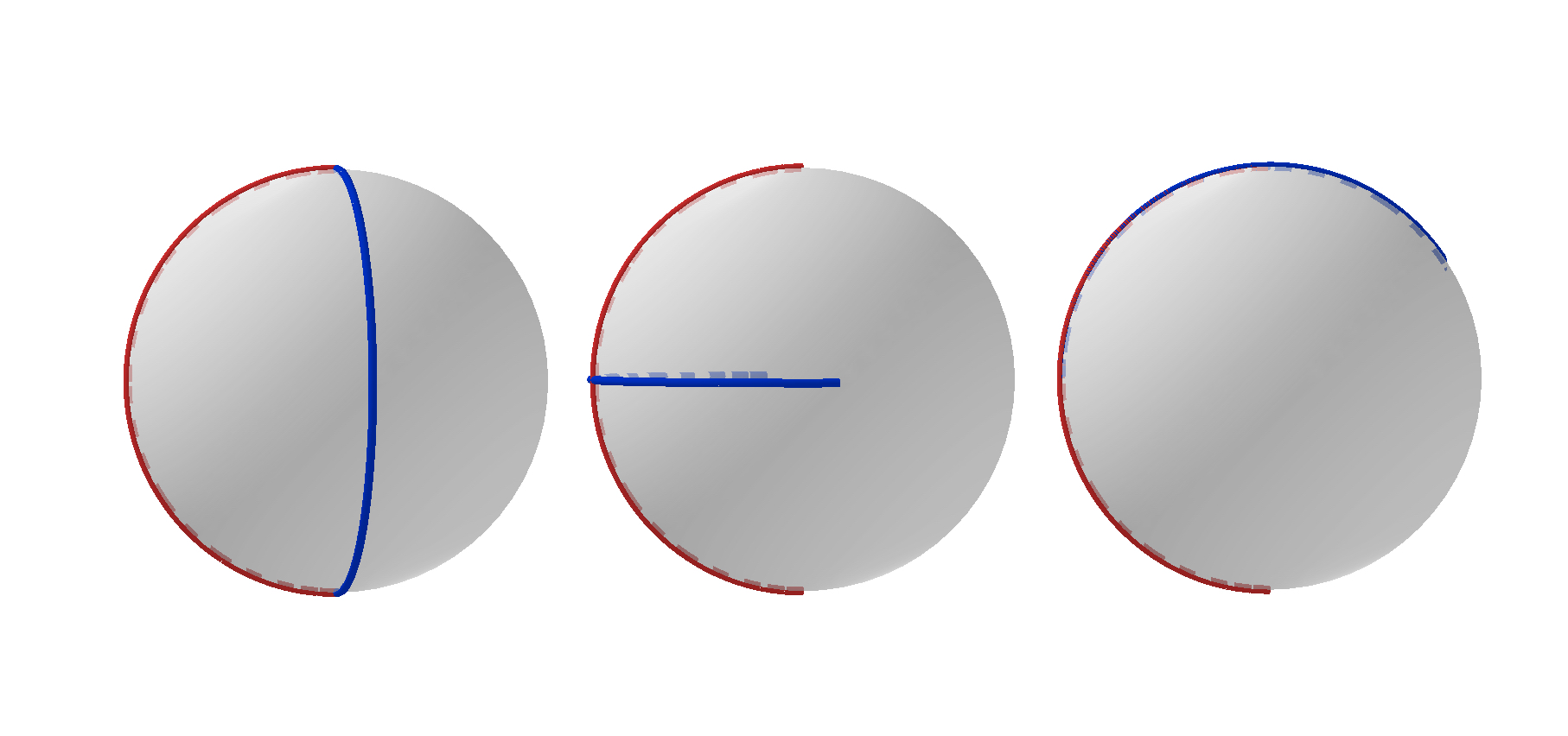}
\caption{\small Situations with non-empty intersection of $\tilde{\Sigma}$ in the sphere.}\label{SigmaSphere}
\end{figure}

Continuing with the above discussion, consider now $u_0$ and $u_1$ solutions of the problem \eqref{eq:EDO from EP1}  in $\Sigma_0$ and $\Sigma_1$ respectively. Only in the case $ \Sigma_0\cap \Sigma_1$ has a non-empty interior, we can construct an initial value problem for which both solutions coincide in the interior of $\Sigma_0\cap \Sigma_1$. In particular, even though the problem \eqref{eq:EDO from EP1} is not defined in the north pole, we obtain a {\it $C^2$ extension} of $u_0$ to the north pole (not just continuous in the north pole). Similarly, we have a $C^2$ extension at the south pole, which finally yields a classical solution of class $C^2$ (extended by the $G$-action) on the whole $M$. This situation could be illustrated in the sphere because the equation \eqref{eq:EDO from EP1} is {\it the same} for $\Sigma_0$ and $\Sigma_1$ as the function $A$  is independent of the center point in the sphere. As mentioned before, this situation is common for any two-point homogeneous space (cf. 4. in Example \ref{ex:polaract} where we show a $\tilde{\Sigma}$ with such properties). Summarizing, we have the following corollary:

\begin{corollary}\label{Corotwo-point homogeneous space}
    Let $M$ be a two-point homogeneous space 
    and $f\in C(M\times \mathbb{R})$  that satisfies conditions [F0], for which  any transverse manifold $\Sigma$ of dimension 1, and any initial conditions $u(q)$, $u'(q)$,  $q\in \Sigma$, there exists a solution $u$ in $C^2(\Sigma)$. If $p_0\in M$ is a fixed point of the polar action with transverse submanifold $\Sigma$, then the problem \eqref{EP1} has a radial $C^2$ solution on the whole $M_0\cup \{p_0\}$.
\end{corollary}

Note that the previous condition on $M$ excludes cases as the paraboloid $z=x^2+y^2$ (and other similar revolution surfaces) because the function $A$ is different when considering the center in (0,0,0) and other different points $(p_1,p_2,p_1^2+p_2^2)$. Recall that this manifold is not a two-point homogeneous space.
\begin{proof}[Proof of Corollary \ref{Corotwo-point homogeneous space}]
    The proof follows the same line of argument as before the statement of the results. Let $u_0$ be a solution of \eqref{eq:EDO from EP1} on $\Sigma_0$ be a transverse manifold with origin at $p_0$. We construct $\Sigma_1$ as the transverse manifold with origin in an $\epsilon$-closed $p_1\in \Sigma$ of a fixed point $p_0$ of the action on $\Sigma_0$. Now, consider $u_1$ be the global solution on $\Sigma_1$ of the ODE as in \eqref{eq:EDO from EP1} but with initial conditions 
    $$u(q_1)=u_0(q_1),\quad u'(q_1)=u'_0(q_1)$$
    for some fixed $q_1\in \mbox{int}(\Sigma_0\cap \Sigma_1)$, which is guaranteed by condition [F0], and the existence hypothesis in the statement of the corollary. By uniqueness of solutions of initial value problem for ODE, we get $u_1=u_0$, when restricted to $\Sigma_0\cap \Sigma_1$. Consequently, this yields a smooth extension of $u_0$ to the fixed point $p_0$.
\end{proof}


\section{Uniqueness of solutions}\label{sec:uniq}

In what follows, we will use previous geometrical tools to establish the uniqueness results stated in Theorem \ref{TheoMainUniq}. We will divide the proof into parts i), ii), and iii) of Theorem \ref{TheoMainUniq}.

\begin{proof}[Proof of Theorem \ref{TheoMainUniq} i)] Let $-\infty\leq c_1<c_2\leq \infty$ be such that $s=J(r):= \int_{r_0}^r \frac{1}{A(t)}\, dt$  is an homeomorphism of class $C^2$ from parametrization of $\Sigma$ into the interval $(c_1,c_2)$.  Since $u_1$ and $u_2$ are maximal radial solutions of \eqref{EP1}, by passing to polar coordinates and using the ODE \eqref{ODEeq}, we have that $z_j(s)=u_j(J^{-1}(s))\in C^2((c_1,c_2))$ satisfies
\begin{equation}\label{ODEeqj}
    \begin{aligned}
    z_j''(s)+A(r(s))^2f(r(s),z_j(s))=0,
\end{aligned}
\end{equation}
for each $j=1,2$. Now, we remark that assumption [U1] implies that there exists $s_0\in (c_1,c_2)$ such that $z_1(s_0)=z_2(s_0)$, and $z'_1(s_0)=z'_2(s_0)$. To justify this, since $u_j$, $j=1,2$, are radial functions (i.e., invariant on the $G$ direction), it follows $\nabla u_j|_p=u_j'(p)\in T_p\Sigma$. It turns out that $z_{1,2}:=z_1-z_2$ solves 
\begin{equation}\label{ODE2}
\left\{ \begin{aligned}
&z_{1,2}''(s)+A(r(s))^2\big(f(r(s),z_1(s))-f(r(s),z_2(s))\big)=0, \\
&z_{1,2}(s_0)=0,  \\
&z'_{1,2}(s_0)=0.
 \end{aligned}\right.   
\end{equation}
At this point, the proof of part i) of Theorem \ref{TheoMainUniq} reduces to showing that zero is the only solution to the above initial value problem. This result depends on the hypotheses assigned to $f$. Thus, to exemplify our arguments, we have decided to work with the condition [F1], from which simple arguments yield uniqueness. To see that this is the case, since $z_{1,2}\in C^2((c_1,c_2))$ is a solution of \eqref{ODE2}, it follows for $s, s_1\in (c_1,c_2)$ that
\begin{equation*}
\begin{aligned}
&z'_{1,2}(s)=z'_{1,2}(s_1)+\int_{s_1}^s A(r(t))^2\big(f(r(t),z_1(t))-f(r(t),z_2(t))\big)\, dt, \\
&z_{1,2}(s)=z_{1,2}(s_1)+\int_{s_1}^s z'_{1,2}(t)\ dt.    
\end{aligned}
\end{equation*}
Thus, let $c_1<R_1<R_2<c_2$. Continuity yields the existence of some $\kappa>0$ such that
\begin{equation}\label{boundcondgeneral}
   \sup_{t\in [R_1,R_2]} |z_j(t)| \leq \kappa,
\end{equation}
for each $j=1,2$. Now, given $\delta>0$ and $s,s_1\in \mathcal{A}(s_1,\delta):=[s_1-\delta,s_1+\delta]\cap[R_1,R_2]$, 
we have from previous identities and Lipschitz condition [F1] that
\begin{equation*}
\begin{aligned}
|z_{1,2}(s)|&+|z'_{1,2}(s)|\\
\leq & |z_{1,2}(s_1)|+|z'_{1,2}(s_1)|\\
&+\Big|\int_{s_1}^s \Big(A(r(t))^2\big(f(r(t),z_1(t))-f(r(t),z_2(t))\big)+z'_{1,2}(t)\Big)\, dt\Big|\\
\leq & |z_{1,2}(s_1)|+|z'_{1,2}(s_1)|\\
&+(1+M)|s-s_1|\sup_{t\in \mathcal{A}(s_1,\delta)}\big(|z_{1,2}(t)|+|z'_{1,2}(t)|\big),
\end{aligned}
\end{equation*}
where 
\begin{equation*}
 M:= C_{R_1,R_2,\kappa}\sup_{t\in [R_1,R_2]} |A(r(t))|^2, 
\end{equation*} 
and given $\kappa>0$ as in \eqref{boundcondgeneral}, $C_{R_1,R_2,\kappa}>0$ is such that
\begin{equation*}
\begin{aligned}
\sup_{\substack{t\in [R_1,R_2], \,  y_1,y_2\in [0,\kappa]\\ y_1\neq y_2}} \frac{|f(r(t),y_1)-f(r(t),y_2)|}{|y_1-y_2|}\leq C_{R_1,R_2,\kappa}.   
\end{aligned}    
\end{equation*}
Notice that [F1] justifies the existence of the previous constant. 

Consequently, setting $\delta=\frac{1}{2(1+M)}$, if $|s-s_1|\leq \delta$, and $z_{1,2}(s_1)=z'_{1,2}(s_1)=0$, the above estimate allows us to conclude that $z_{1,2}(s)=z'_{1,2}(s)=0$ in $[s_1-\delta,s_1+\delta]\cap [R_1,R_2]$. Now, since $\delta$ does not depend on $s_1$, and by \eqref{ODE2} $z_{1,2}(s_0)=z'_{1,2}(s_0)=0$, starting with $s_1=s_0$, we can iterate the previous argument a finite number of times until we obtain $z_{1,2}\equiv 0$ in all $[R_1,R_2]$. Thus, given that $R_1,R_2$ are arbitrary numbers with $c_1<R_1<R_2<c_2$, we conclude that $z_{1,2}=0$ in $(c_1,c_2)$, i.e., $u_1=u_2$ in $M_0$ as desired.
\end{proof}

By taking $u_2=0$, the proof of Theorem \ref{TheoMainUniq} part ii) follows the same arguments above, to avoid repetitions, we omit its deduction. 

\begin{proof}[Proof of Theorem \ref{TheoMainUniq} iii)]
The proof follows similar ideas in the deduction of Theorem \ref{TheoMainUniq} i). Indeed, we consider $(R_1,R_2)\subset (c_1,c_2)$, where $c_1,c_2$ are given as in [G2]. Thus, in polar coordinates  $z_j(s)=u_j(J^{-1}(s))>0$, then it follows by continuity of $z_j$ that there exists $0<\kappa_1<\kappa_2$ such that
\begin{equation*}
   \kappa_1\leq z_j(s)\leq k_2,  
\end{equation*}
for all $s\in [R_1,R_2]$, and all $j=1,2$. On the other hand, by Lipschitz condition [F1], there exists $C_{R_1,R_2,\kappa_1,\kappa_2}>0$ such that
\begin{equation*}
\begin{aligned}
\sup_{\substack{t\in [R_1,R_2], \,  y_1,y_2\in [\kappa_1,\kappa_2]\\ y_1\neq y_2}} \frac{|f(r(t),y_1)-f(r(t),y_2)|}{|y_1-y_2|}\leq C_{R_1,R_2,\kappa_1,\kappa_2},   
\end{aligned}    
\end{equation*}
from which we get
\begin{equation*}
    \begin{aligned}
|f(r(s),z_1(s))&-f(r(s),z_2(s))|\\
&\leq C_{R_1,R_2,\kappa_1,\kappa_2}|z_1(s)-z_2(s)|,     
    \end{aligned}
\end{equation*}
for all $s\in [R_1,R_2]$. At this point, the proof follows by repetition of the arguments in the deduction of Theorem \ref{TheoMainUniq} i) above.
\end{proof}


\section*{Acknowledgments}

The authors would like to express their gratitude to E. Becerra and J. Galvis for their valuable suggestions and comments. They also thank Universidad Nacional de Colombia, Bogot\'a for financial support.


\bibliographystyle{abbrv}

\bibliography{bibli}

\end{document}